\documentclass[12pt]{article}
\title{Generators of invariant linear system on tropical curves for finite isometry group}
\author{Song JuAe}
\date{}

\usepackage{amsmath,amssymb,amscd}
\usepackage{amsthm}

\newtheorem{dfn}{Definition}[subsection] 

\newtheorem{lemma}[dfn]{Lemma}
\newtheorem{rem}[dfn]{Remark}

\newtheorem{dfn3}{Definition}[section] 
\newtheorem{thm3}[dfn3]{Theorem}

\newtheorem{cor3}[dfn3]{Corollary}
\newtheorem{lemma3}[dfn3]{Lemma}
\newtheorem{rem3}[dfn3]{Remark}
\newtheorem{ex3}[dfn3]{Example}

\def\Gamma{\varGamma}

\begin{document}
\maketitle 
\begin{abstract} 
For a tropical curve $\Gamma$ and a finite subgroup $K$ of the isometry group of $\Gamma$, we prove, extending the work by Haase, Musiker and Yu (\cite{Haase=Musiker=Yu}), that the $K$-invariant part of the complete linear system associated to a $K$-invariant effective divisor on $\Gamma$ is finitely generated. 
\end{abstract}


\section{Introduction} 

Let $R(D)$ denote the set consisting of rational functions corresponding to the complete linear system $|D|$ for an effective divisor $D$ on a tropical curve $\Gamma$, where a tropical curve means a metric graph possibly with unbounded edges.
$R(D)$ becomes a tropical semimodule.
The projective space $R(D) / \boldsymbol{R}$ is naturally identified with the complete linear system $|D|$.
Haase, Musiker and Yu showed that $R(D)$ is finitely generated ({\cite[Theorem 6]{Haase=Musiker=Yu}}).
A tropical subsemimodule $R^{\prime}$ of $R(D)$ corresponds to a linear subspace $\Lambda$ of $|D|$.
This linear subspace $\Lambda$ is called a linear system associated to $R^{\prime}$.

In this paper, we recall some basic facts of tropical curves in Section 2.
Then in Section 3, we observe the $K$-invariant set $R(D)^K$ of $R(D)$ and prove that $R(D)^K$ is actually finitely generated, where $K$ is a finite subgroup of the isometry group of $\Gamma$.
Our proof is basically analogous to that of \cite{Haase=Musiker=Yu}, but it is not perfectly compatible, i.e. the $K$-invariant set $S^K$ of the generator set $S$ of $R(D)$ defined in {\cite[Lemma 6]{Haase=Musiker=Yu}} is not a generator set of $R(D)^K$.
We find such a set corresponding to $S$, which we call $S(D)_K$.
The condition defining $S(D)_K$ is tangibly given from geometric information.
Also, the construction of a harmonic morphism with degree $|K|$ from $\Gamma$ to the quotient tropical curve $\Gamma^{\prime}$ of $\Gamma$ by $K$ precedes.
We follow Chan's natural construction (\cite{Chan}) with a little bit of adaptation.
Finally, using the harmonic morphism, we prove that $R(D)^K$ is finitely generated as a tropical semimodule in our main theorem (Theorem \ref{main theorem}).
When $D$ is $K$-invariant, we can identify $R(D)^K/\boldsymbol{R}$ with the $K$-invariant linear subsystem $|D|^K$ and then $|D|^K$ is finitely generated by $S(D)_K/\boldsymbol{R}$.

\section{Preliminaries}　

In this section, we briefly recall the theories of tropical curves (\cite{Kawaguchi=Yamaki}), divisors on tropical curves (\cite{Kawaguchi=Yamaki}), harmonic morphisms of tropical curves (\cite{Chan}, \cite{Haase=Musiker=Yu}, \cite{Kageyama}), and chip-firing moves on tropical curves (\cite{Haase=Musiker=Yu}), which we need later.

\subsection{Tropical curves} 

In this paper, a {\it graph} means an unweighted, finite connected nonempty multigraph.
Note that we allow the existence of loops.
For a graph $G$, the sets of vertices and edges are denoted by $V(G)$ and $E(G)$, respectively.
The {\it valence} ${\rm val}(v)$ of a vertex $v$ of $G$ is the number of edges emanating from $v$, where we count each loop as two.
A vertex $v$ of $G$ is a {\it leaf end} if $v$ has valence one.
A {\it leaf edge} is an edge of $G$ having a leaf end.

An {\it edge-weighted graph} $(G, l)$ is the pair of a graph $G$ and a function $l: E(G) \rightarrow {\boldsymbol{R}}_{>0} \cup \{\infty\}$ called a {\it length function}, where $l$ can take the value $\infty$ on only leaf edges.
A {\it tropical curve} is the underlying $\infty$-metric space of an edge-weighted graph $(G, l)$.
For a point $x$ on a tropical curve $\Gamma$ obtained from $(G, l)$, if the distances between $x$ and all points on $\Gamma$ other than $x$ are infinity, then $x$ is called a {\it point at infinity}, else, $x$ is said to be a {\it finite point}.
For the above tropical curve $\Gamma$, $(G, l)$ is said to be its {\it model}.
There are many possible models for $\Gamma$.
We construct a model $(G_{\circ}, l_{\circ})$ called the {\it canonical model} of $\Gamma$ as follows:
when $\Gamma$ is a circle, we determine $V(G_{\circ})$ as the set consisting of one arbitrary point on $\Gamma$,
else when $\Gamma$ is the $\infty$-metric space obtained from only one edge with length of $\infty$, $V(G_{\circ})$ consists of the two endpoints of $\Gamma$ (those are points at infinity) and an any point on $\Gamma$ as the origin,
else, generally, we determine $V(G_{\circ}):= \{ x \in \Gamma~|~{\rm val}(x) \neq 2 \}$, where the valence ${\rm val}(x)$ is the number of connected components of $U_x \setminus \{ x \}$ with $U_x$ being any sufficiently small connected neighborhood of $x$ in $\Gamma$.
Since connected components of $\Gamma \setminus V(G_{\circ})$ consist of open intervals, whose lengths determine the length function $l_{\circ}$.
If a model $(G, l)$ of $\Gamma$ has no loops, then $(G, l)$ is said to be a {\it loopless model} of $\Gamma$.
For a model $(G, l)$ of $\Gamma$, the loopless model for $(G, l)$ is obtained by regarding all midpoints of loops of $G$ as vertices and by adding them to the set of vertices of $G$.
The loopless model for the canonical model of a tropical curve is called the {\it canonical loopless model}.

For terminology, in a tropical curve $\Gamma$, an edge of $\Gamma$ means an edge of the underlying graph $G_{\circ}$ of the canonical model $(G_{\circ}, l_{\circ})$.
Let $e$ be an edge of $\Gamma$ which is not a loop.
We regard $e$ as a closed subset of $\Gamma$, i.e., including the endpoints $v_1, v_2$ of $e$.
The {\it relative interior} of $e$ is $e^{\circ} = e \setminus \{ v_1, v_2 \}$.
For a point $x$ on $\Gamma$, a {\it half-edge} of $x$ is a connected component of $U_x \setminus \{ x \}$ with any sufficiently small connected neighborhood $U_x$ of $x$.

For a model $(G, l)$ of a tropical curve $\Gamma$, we frequently identify a vertex $v$ (resp. an edge $e$) of $G$ with the point corresponding to $v$ on $\Gamma$ (resp. the closed subset corresponding to $e$ of $\Gamma$).


\subsection{Divisors on tropical curves}　

Let $\Gamma$ be a tropical curve.
An element of the free abelian group ${\rm Div}(\Gamma)$ generated by points on $\Gamma$ is called a {\it divisor} on $\Gamma$.
For a divisor $D$ on $\Gamma$, its {\it degree} ${\rm deg}(D)$ is defined by the sum of the coefficients over all points on $\Gamma$.
We write the coefficient at $x$ as $D(x)$.
A divisor $D$ on $\Gamma$ is said to be {\it effective} if $D(x) \ge 0$ for any $x$ in $\Gamma$. If $D$ is effective, we write simply $D \ge 0$.
The set of points on $\Gamma$ where the coefficient(s) of $D$ is not zero is called the {\it support} of $D$ and written as ${\rm supp}(D)$.

A {\it rational function} on $\Gamma$ is a constant function of $-\infty$ or a piecewise linear function with integer slopes and with a finite number of pieces, taking the value $\pm \infty$ only at points at infinity.
${\rm Rat}(\Gamma)$ denotes the set of rational functions on $\Gamma$.
For a point $x$ on $\Gamma$ and $f$ in ${\rm Rat}(\Gamma)$ which is not constant $-\infty$, the sum of the outgoing slopes of $f$ at $x$ is denoted by ${\rm ord}_x(f)$.
If $x$ is a point at infinity and $f$ is infinite there, we define ${\rm ord}_x(f)$ as the outgoing slope from any sufficiently small connected neighborhood of $x$.
Note when $\Gamma$ is a singleton, for any $f$ in ${\rm Rat}(\Gamma)$, we define ${\rm ord}_x(f) := 0$.
This sum is $0$ for all but finite number of points on $\Gamma$, and thus
\[
{\rm div}(f):=\sum_{x \in \Gamma}{\rm ord}_x(f) \cdot x
\]
is a divisor on $\Gamma$, which is called a {\it principal divisor}.
Two divisors $D$ and $E$ on $\Gamma$ are said to be {\it linearly equivalent} if $D-E$ is a principal divisor.
We handle the values $\infty$ and $-\infty$ as follows:
let $f, g$ in ${\rm Rat}(\Gamma)$ take the value $\infty$ and $-\infty$ at a point $x$ at infinity on $\Gamma$ respectively,
and $y$ be any point in any sufficiently small neighborhood of $x$.
When ${\rm ord}_x(f) + {\rm ord}_x(g)$ is negative, then $(f \odot g)(x) := \infty$.
When ${\rm ord}_x(f) + {\rm ord}_x(g)$ is positive, then $(f \odot g)(x) := -\infty$.
Remark that the constant function of $-\infty$ on $\Gamma$ dose not determine a principal divisor.
For a divisor $D$ on $\Gamma$, the {\it complete linear system} $|D|$ is defined by the set of effective divisors on $\Gamma$ being linearly equivalent to $D$.

The set of $\boldsymbol{R}$ with two tropical operations:
\begin{center}
$a \oplus b := {\rm max}\{ a, b \}$~~~and~~~$a \odot b := a + b$
\end{center}
becomes a semiring called the {\it tropical semiring}, where both $a$ and $b$ are in $\boldsymbol{R}$.
For a divisor $D$ on a tropical curve, let $R(D)$ be the set of rational functions $f \ne -\infty$ such that $D + {\rm div}(f)$ is effective.
When ${\rm deg}(D)$ is negative, $|D|$ is empty, so is $R(D)$.
Otherwise, from the argument in Section $3$ of \cite{Haase=Musiker=Yu}, $D$ is not empty and consequently so is $R(D)$.
Hereafter, we treat only divisors of nonnegative degree.

\begin{lemma}[cf. {\cite[Lemma 4]{Haase=Musiker=Yu}}]
	\label{R(D)はトロピカル半加群}
$R(D)$ becomes a tropical semimodule on $\boldsymbol{R}$ by extending above tropical operations onto functions, giving pointwise sum and product.
\end{lemma}

By the definition of ${\rm ord}_x(f)$ for a point $x$ at infinity and $f$ in ${\rm Rat}(\Gamma)$, we can prove Lemma \ref{R(D)はトロピカル半加群} in the same way of {\cite[Lemma 4]{Haase=Musiker=Yu}}.



For a tropical subsemimodule $M$ of $(\boldsymbol{R} \cup \{ \pm \infty \})^{\Gamma}$ (or of $\boldsymbol{R}^{\Gamma}$), $f$ in $M$ is called an {\it extremal of} $M$ when it implies $f = g_1$ or $f = g_2$ that any $g_1$ and $g_2$ in $M$ satisfies $f = g_1 \oplus g_2$.

\begin{rem}[{\cite[Proposition 8]{Haase=Musiker=Yu}}]
Any finitely generated tropical subsemimodule $\widetilde{M}$ of $\boldsymbol{R}^{\Gamma}$ is generated by the extremals of $\widetilde{M}$.
\end{rem}

With the adaptation for $\pm \infty$, we can prove the following lemma in same way as the above remark.

\begin{lemma}
	\label{有限生成トロピカル半加群はextremalで生成される}
Any finitely generated tropical subsemimodule $M$ of $R(D) \subset (\boldsymbol{R} \cup \{ \pm \infty \})^{\Gamma}$ is generated by the extremals of $M$.
\end{lemma}




\subsection{Harmonic morphisms}　

Let $\Gamma, \Gamma^{\prime}$ be tropical curves, respectively, and $\varphi : \Gamma \rightarrow \Gamma^{\prime}$ be a continuous map.
The map $\varphi$ is called a {\it morphism} if there exist a model $(G, l)$ of $\Gamma$ and a model $(G^{\prime}, l^{\prime})$ of $\Gamma^{\prime}$ such that the image of the set of vertices of $G$ by $\varphi$ is a subset of the set of vertices of $G^{\prime}$, the inverse image of the relative interior of any edge of $G^{\prime}$ by $\varphi$ is the union of the relative interiors of a finite number of edges of $G$ and the restriction of $\varphi$ to any edge $e$ of $G$ is a dilation by some non-negative integer factor ${\rm deg}_e(\varphi)$.
Note that the dilation factor on $e$ with ${\rm deg}_e(\varphi) \ne 0$ represents the ratio of the distance of the images of any two points $x$ and $y$ except points at infinity on $e$ to that of original $x$ and $y$.
If an edge $e$ is mapped to a vertex of $G^{\prime}$ by $\varphi$, then ${\rm deg}_e(\varphi) = 0$.
The morphism $\varphi$ is said to be {\it finite} if ${\rm deg}_e(\varphi) > 0$ for any edge $e$ of $G$.
For any half-edge $h$ of any point on $\Gamma$, we define ${\rm deg}_h(\varphi)$ as ${\rm deg}_e(\varphi)$, where $e$ is the edge of $G$ containing $h$.

Let $\Gamma^{\prime}$ be not a singleton and $x$ a point on $\Gamma$.
The morphism $\varphi$ is {\it harmonic at} $x$ if the number
\[
{\rm deg}_x(\varphi) := \sum_{h \mapsto h^{\prime}}{\rm deg}_h(\varphi)
\]
is independent of the choice of half-edge $h^{\prime}$ emanating from $\varphi(x)$, where $h$ is a connected component of the inverse image of $h^{\prime}$ by $\varphi$. The morphism $\varphi$ is {\it harmonic} if it is harmonic at all points on $\Gamma$.
One can check that if $\varphi$ is a finite harmonic morphism, then the number
\[
{\rm deg}(\varphi) := \sum_{x \mapsto x^{\prime}}{\rm deg}_x(\varphi)
\]
is independent of the choice of a point $x^{\prime}$ on $\Gamma^{\prime}$, and is said the {\it degree} of $\varphi$, where $x$ is an element of the inverse image of $x^{\prime}$ by $\varphi$.
If $\Gamma^{\prime}$ is a singleton and $\Gamma$ is not a singleton, for any point $x$ on $\Gamma$, 
we define ${\rm deg}_x(\varphi)$ as zero so that we regard $\varphi$ as a harmonic morphism of degree zero.
If both $\Gamma$ and $\Gamma^{\prime}$ are singletons, we regard $\varphi$ as a harmonic morphism which can have any number of degree.

Let $\varphi : \Gamma \rightarrow \Gamma^{\prime}$ be a harmonic morphism between tropical curves.
For $f$ in ${\rm Rat}(\Gamma)$, the {\it push-forward} of $f$ is the function $\varphi_\ast f: \Gamma^{\prime} \rightarrow \boldsymbol{R} \cup \{ \pm \infty \}$ defined by 
\[
\varphi_\ast f(x^{\prime}) := \sum_{\substack{x \in \Gamma \\ \varphi(x) = x^{\prime}}} {\rm deg}_x(\varphi) \cdot f(x).
\]
The {\it pull-back} of $f^{\prime}$ in ${\rm Rat}(\Gamma^{\prime})$ is the function $\varphi^{\ast}f^{\prime} : \Gamma \rightarrow \boldsymbol{R} \cup \{ \pm \infty \}$ defined by $\varphi_{\ast}f^{\prime} := f^{\prime} \circ \varphi$.
We define the {\it push-forward} on divisors $\varphi_\ast : {\rm Div}(\Gamma) \rightarrow {\rm Div}(\Gamma^{\prime})$ by 
\[
\varphi_\ast (D) := \sum_{x \in \Gamma}D(x) \cdot \varphi(x).
\]
One can check that ${\rm deg}(\varphi_\ast(D)) = {\rm deg}(D)$ and $\varphi_\ast ({\rm div}(f)) = {\rm div}(\varphi_\ast f)$ for any divisor $D$ on $\Gamma$ and any $f$ in ${\rm Rat}(\Gamma)$ (cf. \cite[Proposition 4.2]{Baker=Norine}).

\subsection{Chip-firing moves}　

In \cite{Haase=Musiker=Yu}, Haase, Musiker and Yu used the term {\it subgraph} of a tropical curve as a compact subset of the tropical curve with a finite number of connected components and  defined the {\it chip firing move} ${\rm CF}(\widetilde{\Gamma_1}, l)$ by a subgraph $\widetilde{\Gamma_1}$ of a tropical curve $\widetilde{\Gamma}$ and a positive real number $l$ as the rational function ${\rm CF}(\widetilde{\Gamma_1}, l)(x) := - {\rm min}(l, {\rm dist}(x, \widetilde{\Gamma_1}))$, where ${\rm dist}(x, \widetilde{\Gamma_1})$ is the infimum of the lengths of the shortest path to arbitrary points on $\widetilde{\Gamma_1}$ from $x$.
They proved that every rational function on a tropical curve is an (ordinary) sum of chip firing moves (plus a constant) ({\cite[Lemma 2]{Haase=Musiker=Yu}}) with the concept of a {\it weighted chip firing move}.
This is a rational function on a tropical curve having two disjoint proper subgraphs $\widetilde{\Gamma_1}$ and $\widetilde{\Gamma_2}$ such that the complement of the union of $\widetilde{\Gamma_1}$ and $\widetilde{\Gamma_2}$ in $\widetilde{\Gamma}$ consists only of open line segments and such that the rational function is constant on $\widetilde{\Gamma_1}$ and $\widetilde{\Gamma_2}$ and linear (smooth) with integer slopes on the complement.
A weighted chip firing move is an (ordinary) sum of chip firing moves (plus a constant) ({\cite[Lemma 1]{Haase=Musiker=Yu}}).

With unbounded edges, their definition of chip firing moves needs a little correction.
Let $\Gamma_1$ be a subgraph of a tropical curve $\Gamma$ which does not have any connected components consisting only of points at infinity and $l$ a positive real number or infinity.
The {\it chip firing move} by $\Gamma_1$ and $l$ is defined as the rational function ${\rm CF}(\Gamma_1, l)(x) := - {\rm min}(l, {\rm dist}(x, \Gamma_1))$.

\begin{lemma}
A weighted chip firing move on a tropical curve is a linear combination of chip firing moves having integer coefficients (plus a constant).
\end{lemma}

\begin{proof}[Sketch of proof.]
We use the same notations as in their proof.
All we have to do is to show the construction for the case with $l = \infty$.
Especially, it is sufficient to check the case that $\Gamma_1$ consists only of points at infinity.
Supposing that $\Gamma_1$ has only one point gives only two situations.
Firstly, $\Gamma_2$ contains a finite point.
Then $f$ can be written as $\pm s \cdot {\rm CF}(\Gamma_2, \infty)$ plus a constant, where $s$ is the slope of $f$ on the complement.
Secondly, $\Gamma_2$ consists only of one point at infinity.
Taking a finite point $x$, then $f$ can be written as $\pm s \cdot ({\rm CF}(f^{-1}([f(x), \infty]), \infty) - {\rm CF}(\{x\}, \infty))$ plus a constant with same $s$ as the first situation.
Suppose that $\Gamma_1$ has plural points.
$\Gamma_2$ must contain at least one finite point.
Let $x_i$ be the intersection of $\Gamma_1$ and the closure of $L_i$.
Note that $\Gamma_1 = \{ x_1, \cdots, x_k \}$, where $k$ is no less than two.
With the slope $s_i$ of $f$ on $e_i := L_i \sqcup \{ x_i \}$, $f$ is $\sum_{i = 1}^{k}({\pm s_i \cdot {\rm CF}(\Gamma \setminus e_i, \infty)})$ plus a constant.
\end{proof}

The next lemma is proven in the same way of {\cite[Lemma 2]{Haase=Musiker=Yu}} and shows the appropriateness of this definition.

\begin{lemma}
Every rational function on a tropical curve is a linear combination of chip firing moves having integer coefficients (plus a constant).
\end{lemma}


A point on $\Gamma$ with valence two is said to be a {\it smooth} point.
We sometimes refer to an effective divisor $D$ on $\Gamma$ as a {\it chip configuration}.
We say that a subgraph $\Gamma_1$ of $\Gamma$ can {\it fire on} $D$ if for each boundary point of $\Gamma_1$ there are at least as many chips as the number of edges pointing out of $\Gamma_1$.
A set of points on a tropical curve $\Gamma$ is said to be {\it cut set} of $\Gamma$ if the complement of that set in $\Gamma$ is disconnected.


\section{Generators of $R(D)^K$}　

In this section, for an effective divisor $D$ on a tropical curve and a finite subgroup $K$ of the isometry group of the tropical curve, we find a generator set of the $K$-invariant set $R(D)^K$ of $R(D)$ and then, show that $R(D)^K$ is finitely generated as a tropical semimodule.
When $D$ is $K$-invariant, $R(D)/\boldsymbol{R}$ is identified with the $K$-invariant linear system $|D|^K$, so $|D|^K$ is finitely generated by the generators of $R(D)^K$ modulo tropical scaling.

\begin{rem3}[{\cite[Lemma 6]{Haase=Musiker=Yu}}]
	\label{R(D)の有限生成性}
Let $\widetilde{\Gamma}$ be a tropical curve, $\widetilde{D}$ be a divisor on $\widetilde{\Gamma}$ and $S$ be the set of rational functions $f$ in $R(\widetilde{D})$ such that the support of $\widetilde{D} + {\rm div}(f)$ does not contain any cut set of $\widetilde{\Gamma}$ consisting only of smooth points.
Then
\begin{itemize}
\item[$(1)$]
$S$ contains all the extremals of $R(\widetilde{D})$,
\item[$(2)$]
$S$ is finite modulo tropical scaling, and
\item[$(3)$]
$S$ generates $R(\widetilde{D})$ as a tropical semimodule.
\end{itemize}
\end{rem3}

Though in the above remark they assume that $R(\widetilde{D})$ is a subset of $\boldsymbol{R}^{\widetilde{\Gamma}}$, the proof is applied even in the case that $R(\widetilde{D})$ is a subset of $(\boldsymbol{R} \cup \{ \pm \infty \})^{\widetilde{\Gamma}}$ with preparations in Section 2.
Also, the above remark throws the relation between $S$ and $\widetilde{D}$ into relief, hence hereafter we write $S$ for $\widetilde{D}$ as $S(\widetilde{D})$.
Note that a tropical subsemimodule of $R(\widetilde{D})$ is not always finitely generated.
Consider the tropical subsemimodule of $R([0])$ corresponding to $|[0]| \setminus \{ [0] \}$ on a tropical curve $[0,1]$.

Let $\Gamma$ be a tropical curve, $D$ an effective divisor on $\Gamma$ and $K$ a subgroup of the isometry group of $\Gamma$.
One can expect the relation between $R(D)$ and $S(D)$ to be analogous to that of their $K$-invariant counterparts $R(D)^K$ and $S(D)^K$, but in vain.
Indeed, the next example objects.

\begin{ex3}
	\label{R(D)^KとS^Kの関係がR(D)とSのそれとは違う例}
Let $\widetilde{\Gamma}$ be a circle and let a map $i: \widetilde{G_0} \rightarrow \widetilde{G_0}$ which transfers two edges to each other, where $\widetilde{G_0}$ is the underlying graph of the canonical loopless model of $\widetilde{\Gamma}$.
For a point $x_1$ on $\widetilde{\Gamma}$, we choose another point $x_2$ on $\widetilde{\Gamma}$ such that $i(x_1) \ne x_2$.
For the group $\widetilde{K}$ generated by $i$ and the effective divisor $\widetilde{D} = x_1 + x_2$, although $S(\widetilde{D})^{\widetilde{K}}$ is empty, $R(\widetilde{D})^{\widetilde{K}}$ is not empty.
It means that $S(\widetilde{D})^{\widetilde{K}}$ is not a generator set of $R(\widetilde{D})^{\widetilde{K}}$.
\end{ex3}

Now, let us find a generator set for $R(D)^K$ that corresponds to $S(D)$ for $R(D)$.
In the above situation, $K$ acts on $\Gamma$ naturally.
We define $V_1(\Gamma)$ as the set of points $x$ on $\Gamma$ such that there exists a point $y$ in any neighborhood of $x$ whose stabilizer is not equal to that of $x$.


\begin{lemma3}
	\label{$V_1$は有限集合}
$V_1(\Gamma)$ is a finite set.
\end{lemma3}


\begin{proof}
We assume that $\Gamma$ is not the $\infty$-metric space obtained from only one edge with length of $\infty$.
Let $\sigma : \Gamma \rightarrow \Gamma$ be an isometry.
Then, for any edge $e$ of $\Gamma$, the image of $e$ by $\sigma$ agrees completely with $e$ or the intersection of $e$ and the image of $e$ by $\sigma$ is contained in the set of the endpoints of $e$.
In fact, if $|e \cap \sigma(e)|$ is infinite, then $\sigma(e)$ is contained in $e$ because $e$ is an edge of $\Gamma$.
It means that $\sigma(e) = e$.
If $|e \cap \sigma(e)|$ is finite and $e \cap \sigma(e)$ contains a point on $\Gamma$ other than endpoints of $e$, then that point has the valence of greater than two.
It contradicts to the fact that $e$ is an edge of $\Gamma$.

From the above argument, for any edge $e$ of $\Gamma$, we can roughly classify the situations into four.
First, $\sigma$ is the identity map on $e$, i.e., $\sigma$ fixes all points on $e$.
Second, $\sigma$ gives a mirror image of $e$.
In this case, if $\Gamma$ is a circle consisting of $e$, the fixed points on $e$ by $\sigma$ are only antipodal points on the axis of symmetry of $\sigma$, otherwise, the midpoint of $e$ is fixed by $\sigma$, moreover when $e$ is a loop, then the vertex connected to $e$ is also fixed by $\sigma$.
Third, $\sigma$ acts as a proper rotation on $e$.
This is possible only when $\Gamma$ is a circle, and $\sigma$ gives no fixed points on $e$.
Finally, $\sigma$ maps $e$ onto other edge of $\Gamma$, then only the endpoints of $e$ may be fixed by $\sigma$.

Consequently, under the above assumption, since $K$ is a finite set and $\Gamma$ has finite vertices and edges, $V_1(\Gamma)$ is a finite set.

Let us suppose that $\Gamma$ is the $\infty$-metric space obtained from only one edge with length of $\infty$.
Since $K$ is a finite set, any $\sigma$ in $K$ is not a proper translation of $\Gamma$.
Each isometry of $\Gamma$ other than translations fixes only one point on $\Gamma$.
Thus, also in this case, $V_1(\Gamma)$ is a finite set.
Note that there can exists only one inversion.
If there were two distinct, these two can generate a translation, leading $|K|$ to infinity.
\end{proof}

We set $(G_0, l_0)$ as the canonical loopless model of $\Gamma$.
By Lemma \ref{$V_1$は有限集合}, we obtain the model $(\widetilde{G_1}, \widetilde{l_1})$ of $\Gamma$ by setting the $K$-orbit of the union of $V(G_0)$ and $V_1(\Gamma)$ as the set of vertices $V(\widetilde{G_1})$.
Naturally, we can regard that $K$ acts on $V(\widetilde{G_1})$ and also on $E(\widetilde{G_1})$.
Thus, the sets $V(\widetilde{G^{\prime}})$ and $E(\widetilde{G^{\prime}})$ are defined as the quotient sets of $V(\widetilde{G_1})$ and $E(\widetilde{G_1})$ by $K$, respectively.
Let $\widetilde{G^{\prime}}$ be the graph obtained by setting  $V(\widetilde{G^{\prime}})$ as the set of vertices and $E(\widetilde{G^{\prime}})$ as the set of edges.
Since $\widetilde{G_1}$ is connected, $\widetilde{G^{\prime}}$ is also connected.
We obtain the loopless graph $G^{\prime}$ from $\widetilde{G^{\prime}}$ and the loopless model $(G_1, l_1)$ of $\Gamma$ from the inverse image of $V(G^{\prime})$ by the map defined by $K$.
Note that $V(G_1)$ contains $V(\widetilde{G_1})$.
Since $K$ is a finite subgroup of the isometry group of $\Gamma$, the length function $l^{\prime} : E(G^{\prime}) \rightarrow {\boldsymbol{R}}_{> 0} \cup \{ \infty \}$, $[e] \mapsto |K_e| \cdot l_1(e)$ is well-defined, where $[e]$ and $K_e$ mean the equivalence class of $e$ and the stabilizer of $e$, respectively.
Let $\Gamma^{\prime}$ be the tropical curve obtained from $(G^{\prime}, l^{\prime})$.
Then, $\Gamma^{\prime}$ is the quotient tropical curve of $\Gamma$ by $K$.

For any edge $e$ of $G_1$, by the Orbit-Stabilizer formula, $|K_e|$ is a positive integer.
Thus, for $(G_1, l_1)$ and $(G^{\prime}, l^{\prime})$, there exists only one morphism $\varphi : \Gamma \rightarrow \Gamma^{\prime}$ that satisfies ${\rm deg}_e(\varphi) = |K_e|$ for any edge $e$ of $G_1$.



We obtain the following lemma as an extension of {\cite[Lemma 2.2]{Chan}}.

\begin{lemma3}
If both $\Gamma$ and $\Gamma^{\prime}$ are not singletons, then $\varphi$ is a finite harmonic morphism of degree $|K|$.
\end{lemma3}

\begin{proof}
Clearly, $\varphi$ is finite.
Now we check that $\varphi$ is harmonic and its degree is $|K|$.
Since $K$ is a finite subgroup of the isometry group of $\Gamma$, for any point $x$ on $\Gamma$ and any half-edge $h^{\prime}$ of $\varphi(x)$, each connected component of $\varphi^{-1}(h^{\prime})$ has the same dilation factor ${\rm deg}_h(\varphi)$, where $h$ is a connected component emanating from $x$.
Therefore, for the edge $e$ of $G_1$ containing $h$ and its image $e^{\prime}$ by $\varphi$, the following hold:
\[
{\rm deg}_x(\varphi) = \sum_{\widetilde{h} \mapsto h^{\prime}}{\rm deg}_h(\varphi) = \sum_{\widetilde{e} \mapsto e^{\prime}}{\rm deg}_e(\varphi) =  |Ke| \cdot |K_e| = |K|.
\]
Where $\widetilde{h}$, $\widetilde{e}$ and $Ke$ denote a connected component of $\varphi^{-1}(h^{\prime})$, that of $\varphi^{-1}(e^{\prime})$ and the orbit of $e$ by $K$, respectively.
Note that we use the Orbit-Stabilizer formula at the last equality.
Accordingly, we get the conclusion.
\end{proof}

Note that whether $\Gamma$ is a singleton or not agrees with whether $\Gamma^{\prime}$ is a singleton.

Is $R(D)^K$, the $K$-invariant set of $R(D)$, identical to $\varphi^{\ast}(R(\varphi_{\ast}(D)))$?
Nor is it.

\begin{ex3}
Assume the situation of Example \ref{R(D)^KとS^Kの関係がR(D)とSのそれとは違う例}.
For a rational function $f$ which decreases from $\varphi(x_1)$ to $\varphi(x_2)$ with slope one and is constant on other graph, however $f$ is an element of $R(\varphi_{\ast}(\widetilde{D}))$, the pull-back of $f$ by $\varphi$ is not in $R(\widetilde{D})^{\widetilde{K}}$.
\end{ex3}

Next, for $R(D)^K$, the following holds.

\begin{lemma3}
$R(D)^K$ is a tropical semimodule.
\end{lemma3}

\begin{proof}
Let $c$ be in $\boldsymbol{R}$, $f, g$ in $R(D)^K$ and $\sigma$ in $K$.
Since $R(D)$ is a tropical semimodule by Lemma \ref{R(D)はトロピカル半加群}, $c \odot f$ and $f \oplus g$ are in $R(D)$.
It is obvious that $\odot$ and $\circ$ are associative and that $\circ$ is distributive over $\oplus$ from right, both $(c \odot f) \circ \sigma$ and $(f \oplus g) \circ \sigma$ are in $R(D)^K$.
\end{proof}


Note that $R(D + {\rm div}(f))^K = R(D)^K \odot (-f)$ for any $K$-invariant rational function $f$.





The following lemma is an extension of {\cite[Lemma 5]{Haase=Musiker=Yu}}.

\begin{lemma3}
	\label{extremalであるための必要十分条件}
Let $f$ be in ${\rm Rat}(\Gamma)$.
Then, $f$ is an extremal of $R(D)^K$ if and only if there are not two proper $K$-invariant subgraphs $\Gamma_1$ and $\Gamma_2$ covering $\Gamma$ such that each can fire on $D + {\rm div}(f)$.
\end{lemma3}

\begin{proof}
First, let us show the ``if" part.
Suppose that there are two such subgraphs $\Gamma_1$ and $\Gamma_2$.
We can assume that each $\Gamma_i$ does not have any connected component consisting only of points at infinity.
Each $\Gamma_i$ defines a chip firing move $g_i$ for a small positive number so that $g_i$ is zero on $\Gamma_i$ and they are nonpositive.
As $\Gamma_1$ and $\Gamma_2$ are $K$-invariant, so $g_1$ and $g_2$ are in $R(D + {\rm div}(f))^K$.
Since $g_1 \oplus g_2 = 0$ on $\Gamma$, we can write $f$ as $(f + g_1) \oplus (f + g_2)$, i.e. $f$ is not an extremal of $R(D)^K$.

Next, let us show the ``only if" part.
Suppose $f = g_1 \oplus g_2$ for some $g_1$ and $g_2$ in $R(D)^K \setminus \{ f \}$.
For $i = 1,2$, there exists $\widetilde{g_i}$ in $R(D + {\rm div}(f))^K$ such that $g_i = \widetilde{g_i} \odot f$.
Let $\Gamma_i$ be the closure of the loci where $\widetilde{g_i} = 0$.
Then, the union of $\Gamma_1$ and $\Gamma_2$ is $\Gamma$ and each $\Gamma_i$ is proper.
Since $\widetilde{g_i}$ is $K$-invariant, so is $\Gamma_i$.
Then, each $\Gamma_i$ can fire on $D + {\rm div}(f)$.
\end{proof}

The term ``a subgraph is infinite'' means that the subgraph is a infinite set.

\begin{lemma3}
	\label{部分グラフの存在条件}
Let $A$ be a $K$-invariant subset of ${\rm supp}(D)$.
If $\varphi(A)$ is a cut set of $\Gamma^{\prime}$ and $D(x) \ge {\rm val}(x) - 1$ for any $x$ in $A$, then there exists a $K$-invariant infinite subgraph $\Gamma_1$ of $\Gamma$ which can fire on $D$ and whose boundary points are in $A$.
\end{lemma3}

\begin{proof}
For such $A$, let $\Gamma^{\prime}_1, \cdots, \Gamma^{\prime}_n$ be distinct connected components of $\Gamma^{\prime} \setminus \varphi(A)$ respectively.
Note that $n$ is no less than two since $\varphi(A)$ is a cut set of $\Gamma^{\prime}$.
Clearly, for any $i$, the inverse image of the closure of $\Gamma^{\prime}_i$ by $\varphi$ is a $K$-invariant infinite subgraph of $\Gamma$ we want.
\end{proof}

We call a point on $\Gamma$ not being a vertex of $G_1$ a {\it $K$-ordinary point}.
Note that if a subgraph of $\Gamma$ has a $K$-ordinary point, topologically saying, it should have infinite points.

\begin{lemma3}
	\label{部分グラフの性質}
Let $\Gamma_1$ be a $K$-invariant subgraph of $\Gamma$.
If $\Gamma_1$ is infinite and if the set of its boundary points $\partial \Gamma_1$ contains at least one $K$-ordinary point, then $\varphi(\partial \Gamma_1)$ is a cut set of $\Gamma^{\prime}$ and contains a point on $\Gamma^{\prime}$ not being a vertex of $G^{\prime}$. 
\end{lemma3}

\begin{proof}
For such $\Gamma_1$, obviously $\varphi(\partial \Gamma_1)$ contains a point on $\Gamma^{\prime}$ not being a vertex of $G^{\prime}$.
It is sufficient to check that $\varphi(\partial \Gamma_1)$ is a cut set of $\Gamma^{\prime}$.
Let $\Gamma_2$ be the closure of the complement set of $\Gamma_1$ in $\Gamma$.
This $\Gamma_2$ is $K$-invariant and contains a $K$-ordinary point.
Thus, $\Gamma_2$ is an infinite subgraph.
Consequently, $\Gamma^{\prime} \setminus \varphi(\partial\Gamma_1) = \varphi(\Gamma_1 \cup \Gamma_2) \setminus \varphi(\partial\Gamma_1) = (\varphi(\Gamma_1) \setminus \varphi(\partial\Gamma_1)) \sqcup (\varphi(\Gamma_2) \setminus \varphi(\partial\Gamma_1))$.
Hence, $\varphi(\partial \Gamma_1)$ is a cut set of $\Gamma^{\prime}$.
\end{proof}


The next corollary follows from Lemma \ref{部分グラフの存在条件} and Lemma \ref	{部分グラフの性質}.

\begin{cor3}
	\label{K通常点が動くための条件}
For a subset of the support of $\varphi_\ast(D)$, we consider the following condition $(\ast):$

\begin{itemize}
\item[$(\ast)$]
it is a cut set of $\Gamma^{\prime}$ containing no vertices of $G^{\prime}$ and whose inverse image by $\varphi$ is a subset of the support of $D$.

\item[$(1)$]
For a subset $A$ of ${\rm supp}(D)$ whose image by $\varphi$ satisfies $(\ast)$, there exists a $K$-invariant infinite subgraph $\Gamma_1$ of $\Gamma$ which can fire on $D$ and whose boundary points are in $A$.

\item[$(2)$]
Let $\Gamma_1$ be a $K$-invariant subgraph of $\Gamma$.
If $\Gamma_1$ is infinite and can fire on $D$ and if the set of its boundary points consists only of $K$-ordinary points, then the image of the set of boundary points of $\Gamma_1$ by $\varphi$ satisfies $(\ast)$.
\end{itemize}
\end{cor3}



By Corollary \ref{K通常点が動くための条件}, it is natural to define $S(D)_K$ as the set of $f$ in $R(D)^K$ such that there exist no cut sets of $\Gamma^{\prime}$ contained in the support of $\varphi_\ast(D + {\rm div}(f))$, containing no vertices of $G^{\prime}$ and whose inverse image by $\varphi$ is a subset of the support of $D + {\rm div}(f)$.
In fact, this $S(D)_K$ is the set corresponding to $S(D)$, i.e. $S(D)_K$ is a generator set of $R(D)^K$.

\begin{thm3}
	\label{main theorem}
In the above situation, the following hold:

\begin{itemize}
\item[$(1)$]
$S(D)_K$ contains all the extremals of $R(D)^K$,

\item[$(2)$]
$S(D)_K$ is finite modulo tropical scaling, and

\item[$(3)$]
$S(D)_K$ generates $R(D)^K$ as a tropical semimodule.
\end{itemize}
\end{thm3}

\begin{proof}
$(1)$ Suppose $f$ is in the difference set of $R(D)^K$ from $S(D)_K$, then there exists a cut set $A^{\prime}$ of $\Gamma^{\prime}$ contained in ${\rm supp}(\varphi_\ast(D + {\rm div}(f)))$, containing no vertices of $G^{\prime}$ and such that $\varphi^{-1}(A^{\prime}) \subset {\rm supp}(D + {\rm div}(f))$.
By $(1)$ of Corollary \ref{K通常点が動くための条件}, there exists a $K$-invariant infinite subgraph $\Gamma_1$ of $\Gamma$ which can fire on $D + {\rm div}(f)$ and whose boundary points are in $\varphi^{-1}(A^{\prime})$.
Then, the closure of $\Gamma \setminus \Gamma_1$ can also fire on $D + {\rm div}(f)$.
Therefore, by Lemma \ref{extremalであるための必要十分条件}, $f$ is not an extremal of $R(D)^K$.

$(2)$ The push-forward of a rational function on $\Gamma$ induces a natural map $S(D)_K / \boldsymbol{R} \rightarrow S(\varphi_\ast (D)) / \boldsymbol{R}$, $[f] \mapsto [\varphi_\ast(f)]$.
In fact, for any $f$ in $S(D)_K$, $\varphi_\ast (D + {\rm div}(f)) = \varphi_\ast(D) + \varphi_\ast({\rm div}(f)) = \varphi_\ast(D) + {\rm div}(\varphi_\ast(f))$, thus, $\varphi_\ast(f)$ is in $R(\varphi_\ast(D))$.
From $f \in S(D)_K$, there exist no cut sets of $\Gamma^{\prime}$ contained in ${\rm supp}(\varphi_\ast (D) + \varphi_\ast ({\rm div}(f)))$, containing no vertices of $G^{\prime}$ and whose inverse image by $\varphi$ is a subset of ${\rm supp}(D + {\rm div}(f))$.
This means that $\varphi_\ast(f)$ is in $S(\varphi_\ast (D))$.
Also, for any pair of $f_1$ and $f_2$ in $[f]$, there exists $c$ in $\boldsymbol{R}$ satisfying $f_2 = f_1 + c$.
Since $\varphi_{\ast}(f_2) = \varphi_{\ast}(f_1 + c) =\varphi_{\ast}(f_1) + \varphi_{\ast}(c) = \varphi_{\ast}(f_1) + c$, the map is well-defined.
Now we show that the map is injective.
Let $[f]$ and $[g]$ be distinct elements of $S(D)_K / \boldsymbol{R}$, thus ${\rm div}(f)$ differs from ${\rm div}(g)$.
Since both $f$ and $g$ are $K$-invariant, so their images $\varphi_\ast({\rm div}(f))$ and $\varphi_\ast({\rm div}(g))$ are different, i.e. the map is injective.
By Remark \ref{R(D)の有限生成性}, we get the conclusion.

$(3)$ Suppose $f \in R(D)^K$.
Let $N(f)$ be the number of distinct $K$-orbits in the union of  all $K$-invariant subsets of ${\rm supp}(D + {\rm div}(f))$ which is a cut set of $\Gamma^{\prime}$ containing no vertices of $G^{\prime}$.
We prove $(3)$ by induction for $N(f)$.
If $N(f) = 0$, then $f \in S(D)_K$ from the definition of $S(D)_K$.
Assume that $f \in \langle S(D)_K \rangle$ for all $N(f) \le k$, where $\langle S(D)_K \rangle$ means the tropical semimodule generated by $S(D)_K$.
We consider the case where $N(f) = k + 1$ and $f \notin S(D)_K$.
Let $A$ be a subset of ${\rm supp}(D + {\rm div}(f))$ whose image by $\varphi$ is a cut set of $\Gamma^{\prime}$ containing no vertices of $G^{\prime}$.
By $(1)$ of Corollary \ref{K通常点が動くための条件}, there exists a $K$-invariant subgraph $\Gamma_1$ of $\Gamma$ which can fire on $D + {\rm div}(f)$ and whose boundary points are in $A$.
Let $\Gamma_2$ be the closure of the complement of $\Gamma_1$ in $\Gamma$.
For any $x \in \partial \Gamma_i$, we write the distance between $x$ and its closest vertex of $G_1$ as $l_{x_i}$.
Let $l_i := {\rm min}\{ l_{x_i} | x \in \partial \Gamma_i \}$ and $g_i := {\rm CF}(\Gamma_i, l_i)$.
Then, for both $i = 1, 2$, $f \odot g_i$ is not equal to $f$ and is in $R(D)^K$ since $f, g_i \in R(D)^K$ and $f = (f \odot g_1) \oplus (f \odot g_2)$.
By the definition of $g_i$, $N(f) > N(f \odot g_i)$ and $f \odot g_i \in \langle S(D)_K \rangle$, then $f \in \langle S(D)_K \rangle$.
\end{proof}

By Lemma 	\ref{有限生成トロピカル半加群はextremalで生成される} and the above theorem, we obtain the following corollary, which is an extension of {\cite[Corollary 9]{Haase=Musiker=Yu}}.

\begin{cor3}
Let $\Gamma$ be a tropical curve, $D$ an effective divisor on $\Gamma$ and $K$ a finite subgroup of the isometry group of $\Gamma$.
Then, the tropical semimodule $R(D)^K$ is generated by the extremals of $R(D)^K$.
This generating set is minimal and unique up to tropical scalar multiplication.
\end{cor3}

If $D$ is $K$-invariant, $R(D)^K/\boldsymbol{R}$ is naturally identified with the $K$-invariant linear subsystem $|D|^K$.
In conclusion, the following statement holds from Theorem \ref{main theorem}.

\begin{thm3}
Let $\Gamma$ be a tropical curve, $D$ an effective divisor on $\Gamma$ and $K$ a finite subgroup of the isometry group of $\Gamma$.
If $D$ is $K$-invariant, then the $K$-invariant linear subsystem $|D|^K$ of $|D|$ is finitely generated by $S(D)_K/\boldsymbol{R}$.
\end{thm3}







\begin{thebibliography}{9} 
\bibitem{Baker=Norine} Mathew Baker, Serguei Norine，{\it Harmonic Morphisms and Hyperelliptic Graphs}, Int. Math. Res. Not. {\bf 15} (2009)，2914--2955.
\bibitem{Chan} Melody Chan, {\it Tropical curves and metric graphs}, Univ. of California, Berkeley thesis, 2012.
\bibitem{Haase=Musiker=Yu} Christian Haase, Gregg Musiker, Josephine Yu，{\it Linear Systems on Tropical Curves}, Mathematische Zeitschrift {\bf 270} (2012), 1111--1140.
\bibitem{Kageyama} Yuki Kageyama, {\it Divisorial condition for the stable gonality of tropical curves}, arXiv:1801.07405.
\bibitem{Kawaguchi=Yamaki} Shu Kawaguchi, Kazuhiko Yamaki, {\it Rank of Divisors Under Specialization}，Int. Math. Res. Not. {\bf 12} (2015)，4121--4176.
\bibitem{Shinjo} Mizuho Shinjo, {\it gonality of nonsingular tropical curves with genus three}, master's thesis of Tokyo Metropolitan University, 2017\\
(https://tokyo-metro-u.repo.nii.ac.jp/index.php?action=repository\_view\\
\_main\_item\_detail\&item\_id=6422\&item\_no=1\&page\_id=30\&block\_id=15\\5).
\end{thebibliography}
\end{document}